\documentclass{amsart}
\usepackage{amsmath,amssymb}

\newcommand{\Conv}{\mathrm{Conv}}
\newcommand{\ConvH}{\mathrm{Conv}_{\mathsf H}}
\newcommand{\BConv}{\mathrm{BConv}}

\newcommand{\w}{\omega}
\newcommand{\U}{\mathcal U}
\newcommand{\I}{\mathcal I}
\newcommand{\F}{\mathcal F}

\newcommand{\C}{\mathcal C}

\newcommand{\IR}{\mathbb R}
\newcommand{\IB}{\mathbb B}
\newcommand{\IS}{\mathbb S}

\newcommand{\IN}{\mathbb N}
\newcommand{\dens}{\mathrm{dens}}

\newcommand{\dist}{\mathrm{dist}}

\newcommand{\Cld}{\mathrm{Cld}}
\newcommand{\conv}{\mathrm{conv}}

\newcommand{\dH}{\mathsf{d}_{\mathsf H}\kern-1pt}

\newcommand{\oco}{\overline{\mathrm{conv}}}

\newcommand{\e}{\varepsilon}

\newcommand{\HH}{\mathcal H}
\newcommand{\cconv}{\overline{\conv}}
\newcommand{\pb}{\mathsf{b}^\circ\kern-1pt}

\newtheorem{theorem}{Theorem}

\newtheorem{introcor}{Corollary}
\newtheorem{proposition}{Proposition}[section]
\newtheorem{corollary}[proposition]{Corollary}

\newtheorem{lemma}[proposition]{Lemma}

\theoremstyle{definition}

\newtheorem{remark}[proposition]{Remark}

\newtheorem{problem}[proposition]{Problem}

\title[Classification of the spaces of convex closed sets]{Recognizing the topology of
the space\\ of closed convex subsets of a Banach space}

\author[T. Banakh]{Taras Banakh}
\address[T. Banakh]{Universytet Jana Kochanowskiego, Kielce (Poland) and Ivan Franko National University of Lviv (Ukraine)}
\email{tbanakh@franko.lviv.ua}

\author[I. Hetman]{Ivan Hetman}
\address[I.Getman]{Department of Mathematics, Ivan Franko National University of Lviv, Universytetska 1, Lviv, 79000, Ukraine}
\email{ihromant@gmail.com}

\author[K. Sakai]{Katsuro Sakai}
\address[K. Sakai]{Institute of Mathematics, University of Tsukuba, Tsukuba, 305-8571, Japan}
\email{sakaiktr@sakura.cc.tsukuba.ac.jp}

\keywords{Banach space, space of closed convex sets, Hilbert space, density, Kunen-Shelah property, polyhedral convex set}
\subjclass{57N20; 46A55; 46B26; 46B20; 52B05; 03E65}

\begin{document}
\begin{abstract} Let $X$ be a Banach space and $\ConvH(X)$ be the space of non-empty closed convex subsets of $X$, endowed with the Hausdorff metric $d_H$. We prove that each connected component $\mathcal H$ of the space $\ConvH(X)$ is homeomorphic to one of the spaces: $\{0\}$, $\IR$, $\IR\times\IR_+$, $Q\times\IR_+$, $l_2$, or the Hilbert space $l_2(\kappa)$ of cardinality $\kappa\ge \mathfrak c$. More precisely, a component $\mathcal H$ of $\ConvH(X)$ is homeomorphic to:
\begin{enumerate}
\item $\{0\}$ iff $\mathcal H$ contains the whole space $X$;
\item $\IR$ iff $\mathcal H$ contains a half-space;
\item $\IR\times\bar\IR_+$ iff $\mathcal H$ contains a linear subspace of $X$ of codimension 1;
\item $Q\times\bar\IR_+$ iff $\mathcal H$ contains a linear subspace of $X$ of finite codimension $\ge 2$;
\item $l_2$ iff $\mathcal H$ contains a polyhedral convex subset of $X$ but contains no linear subspace and no half-space in $X$;
\item $l_2(\kappa)$ for some cardinal $\kappa\ge\mathfrak c$ iff $\mathcal H$ contains no polyhedral convex subset of $X$.
\end{enumerate}
\end{abstract}

\maketitle

\section{Introduction}
In this paper we recognize the topological structure of the space $\ConvH(X)$ of non-empty closed convex subsets of a Banach space $X$. The space $\ConvH(X)$ is endowed with the  Hausdorff metric
$$
\dH(A,B)=\max\big\{\sup_{a\in A}\dist(a,B),\sup_{b\in B}\dist(b,A)\big\}\in[0,\infty]$$
where $\dist(a,B)=\inf_{b\in B}\|a-b\|$ is the distance from a point $a$ to a subset $B$ in $X$. In fact, the topology of $\ConvH(X)$ can be defined directly without appealing to the Hausdorff metric: a subset $\U\subset\ConvH(X)$ is open if and only if for every $A\in\U$ there is an open neighborhood $U$ of the origin in $X$ such that $B(A,U)\subset\U$ where $B(A,U)=\{A'\in\ConvH(X):A'\subset A+U$ and $A\subset A'+U\}$. Here as usual, $A+B=\{a+b:a\in A,\; b\in B\}$ stands for the pointwise sum of sets $A,B\subset X$.
In such a way, for every linear topological space $X$ we can define the topology on the space $\ConvH(X)$ of non-empty closed convex subsets of $X$. This topology will be called {\em the uniform topology} on $\ConvH(X)$ because it is generated by the uniformity whose base consists of the sets
$$2^U=\{(A,A')\in\ConvH(X)^2:A\subset A'+U,\; A'\subset A+U\}$$
where $U$ runs over open symmetric neighborhoods of the origin in $X$.

We shall observe in Remark~\ref{r4.6} that for a Banach space $X$ the space $\ConvH(X)$ is locally connected: two sets $A,B\in\ConvH(X)$ lie in the same connected component of $\ConvH(X)$ if and only if $\dH(A,B)<+\infty$. So, in order to understand the topological structure of the hyperspace $\ConvH(X)$ it suffices to recognize the topology of its connected components. This problem is quite easy if $X$ is a 1-dimensional real space. In this case $X$ is isometric to $\IR$ and a connected component $\mathcal H$ of $\ConvH(X)$ is isometric to:
\begin{enumerate}
\item $\{0\}$ iff $X\in\mathcal H$;
\item $\IR$ iff $\mathcal H$ contains a closed ray;
\item $\IR\times\bar\IR_+$ iff $\mathcal H$ contains a bounded set.
\end{enumerate}
Here $\bar\IR_+=[0,+\infty)$ stands for the closed half-line.

For arbitrary Banach spaces we shall add to this list two more spaces:
\begin{itemize}
\item[(4)] $Q\times\bar\IR_+$, where $Q=[0,1]^\w$ is the Hilbert cube;
\item[(5)] $l_2(\kappa)$, the Hilbert space having an orthonormal basis of cardinality $\kappa$.
\end{itemize}
For $\kappa=\w$ the separable Hilbert space $l_2(\w)$ is usually denoted by $l_2$.
By the famous Toru\'nczyk Theorem \cite{Tor81}, \cite{Tor85}, each infinite-dimensional Banach space $X$ of density $\kappa$ is homeomorphic to the Hilbert space $l_2(\kappa)$. In particular, the Banach space $l_\infty$ of bounded real sequences is homeomorphic to $l_2(\mathfrak c)$. In the sequel we shall identify cardinals with the sets of ordinals of smaller cardinality and endow such sets with
discrete topology. The cardinality of a set $A$ is denoted  by
$|A|$.

Let $X$ be a Banach space. As we shall see in Theorem~\ref{main}, each non-locally compact connected component $\HH$ of the space $\ConvH(X)$ is homeomorphic to the Hilbert spaces $l_2(\kappa)$ of density $\kappa=\dens(\HH)$. This reduces the problem of recognizing the topology of $\ConvH(X)$ to calculating the densities of its components. In fact, separable components $\HH$ of $\ConvH(X)$ have been characterized in \cite{BH2} as components containing a polyhedral convex set.

We recall that a convex subset $C$ of a Banach space $X$ is {\em polyhedral} if $C$ can be written as the intersection $C=\bigcap\F$ of a finite family $\F$ of closed half-spaces. A {\em half-space} in $X$ is a convex set of the form $f^{-1}\big((-\infty,a]\big)$ for some real number $a$ and some non-zero linear continuous functional $f:X\to\IR$. The whole space $X$ is polyhedral being the intersection $X=\bigcap\F$ of the empty family $\F=\emptyset$ of closed half-spaces.

The principal result of this paper is the following classification theorem.

\begin{theorem}\label{main} Let $X$ be a Banach space. Each connected component $\mathcal H$ of the space $\ConvH(X)$ is homeomorphic to one of the spaces: $\{0\}$, $\IR$, $\IR\times\bar\IR_+$, $Q\times\bar\IR_+$, $l_2$, or the Hilbert space $l_2(\kappa)$ of density $\kappa\ge\mathfrak c$. More precisely, $\mathcal H$ is homeomorphic to:
\begin{enumerate}
\item $\{0\}$ iff $\mathcal H$ contains the whole space $X$;
\item $\IR$ iff $\mathcal H$ contains a half-space;
\item $\IR\times\bar\IR_+$ iff $\mathcal H$ contains a linear subspace of $X$ of codimension 1;
\item $Q\times\bar\IR_+$ iff $\mathcal H$ contains a linear subspace of $X$ of finite codimension $\ge 2$;
\item $l_2$ iff $\mathcal H$ contains a polyhedral convex subset of $X$ but contains no linear subspace and no half-space in $X$;
\item $l_2(\kappa)$ for some cardinal $\kappa\ge\mathfrak c$ iff $\mathcal H$ contains no polyhedral convex subset of $X$.
\end{enumerate}
\end{theorem}

Theorem~\ref{main} will be proved in Section~\ref{s7} after some preliminary work done in Sections~\ref{s2}--\ref{s5}.

In Corollary~\ref{cor2} below we shall derive from Theorem~\ref{main} a complete topological classification of the spaces $\ConvH(X)$ for Banach spaces $X$ with Kunen-Shelah property and $|X^*|\le\mathfrak c$.

A Banach space $X$ is defined to have the {\em Kunen-Shelah property} if each closed convex subset $C\subset X$ can be written as intersection $C=\bigcap\F$ of an at most countable family $\F$ of closed half-spaces (in fact, this is one of seven equivalent
Kunen-Shelah properties considered in \cite{GJMMP} and \cite[8.19]{HMVZ}).
For a Banach space $X$ with the Kunen-Shelah property we get
$$|X^*|\le|\ConvH(X)|\le|X^*|^\w.$$
The upper bound $\ConvH(X)\le|X^*|^\w$ follows from the definition of the Kunen-Shelah property while the lower bound $|X^*|\le|\ConvH(X)|$ follows from the observation that a functional $f\in X^*$ is uniquely determined by its polar half-space $H_f=f^{-1}\big((-\infty,1]\big)$.

It is clear that each
separable Banach space has the Kunen-Shelah property. However
there are also non-separable Banach spaces with that property. The
first example of such  Banach space was constructed by S.~Shelah
under $\diamondsuit_{\aleph_1}$ \cite{Shel}. The second example is
due to K.~Kunen who used the Continuum Hypothesis to
construct a non-metrizable scattered compact space $K$ such that
the Banach space $X=C(K)$ of continuous functions on $K$ is
hereditarily Lindel\"of in the weak topology and thus has the
Kunen-Shelah Property, see \cite[p.1123]{Neg84}. The Kunen's space $X=C(K)$ has an additional property that its dual space $X^*=C(X)^*$ has cardinality $|X^*|=\mathfrak c$ (this follows from the fact that each Borel measure on the scattered compact space $K$ has countable support). Let us remark
that for every separable Banach space $X$ the dual space $X^*$ also has cardinality of continuum $|X^*|=\mathfrak c$. It should be mentioned that non-separable Banach spaces with the Kunen-Shelah property can be constructed only under certain additional set-theoretic assumptions. By \cite{Tod06}, there are models of ZFC in which each Banach space with the Kunen-Shelah property is separable.

\begin{introcor}\label{cor1} For a separable Banach space \textup{(}more generally, a Banach space with the Kunen-Shelah property and $|X^*|=\mathfrak c$\textup{)}, each connected component $\mathcal H$ of the space $\ConvH(X)$ is homeomorphic to
$\{0\}$, $\IR$, $\IR\times\bar\IR_+$, $Q\times\bar\IR_+$, $l_2$ or $l_\infty$.
More precisely, $\mathcal H$ is homeomorphic to:
\begin{enumerate}
\item $\{0\}$ iff $\mathcal H$ contains the whole space $X$;
\item $\IR$ iff $\mathcal H$ contains a half-space;
\item $\IR\times\bar\IR_+$ iff $\mathcal H$ contains a linear subspace of $X$ of codimension 1;
\item $Q\times\bar\IR_+$ iff $\mathcal H$ contains a linear subspace of $X$ of  codimension $\ge 2$;
\item $l_2$ iff $\mathcal H$ contains a polyhedral convex set but contains no linear subspace and no half-space;
\item $l_\infty$ iff $\mathcal H$ contains no polyhedral convex set.
\end{enumerate}
\end{introcor}

Since the space $\ConvH(X)$ is homeomorphic to the topological sum of its connected components, we can use Corollary~\ref{cor1} to classify topologically the spaces $\ConvH(X)$ for separable Banach spaces $X$ (and more generally Banach spaces with the Kunen-Shelah property and $|X^*|\le\mathfrak c$). In the following corollary the cardinal $\mathfrak c$ is considered as a discrete topological space.

\begin{introcor}\label{cor2} For a separable Banach space $X$ \textup{(}more generally, a Banach space $X$ with the Kunen-Shelah property and $|X^*|\le\mathfrak c$\textup{)} the space $\ConvH(X)$ is homeomorphic to the topological sum:
\begin{enumerate}
\item $\{0\}\oplus \IR\oplus\IR\oplus\,(\IR\times\bar\IR_+)$ iff $\dim(X)=1$;
\item $\{0\}\oplus \,Q\times\bar\IR_+\,\oplus\;\mathfrak c\times(\IR\oplus\,\IR\times\bar\IR_+\,\oplus l_2\oplus l_\infty)$ iff $\dim(X)=2$;
\item $\{0\}\oplus\; \mathfrak c\times(\IR\oplus\,\IR\times\bar\IR_+\,\oplus \,Q\times\bar\IR_+\,\oplus l_2\oplus l_\infty)$ iff $\dim(X)\ge 3$.
\end{enumerate}
Moreover, under $2^{\w_1}>\mathfrak c$, for a Banach space $X$, the space $\ConvH(X)$ has cardinality $|\ConvH(X)|\le\mathfrak c$ if and only if $|X^*|\le\mathfrak c$ and the Banach space $X$ has the Kunen-Shelah property.
\end{introcor}

\begin{proof} The statements (1)--(3) easily follow from the classification of the components of $\Conv_H(X)$ given in Corollary~\ref{cor1} and a routine calculation of the number of components of a given topological type.

Now assume that $2^{\w_1}>\mathfrak c$. If $X$ is a Banach space with the Kunen-Shelah property and $|X^*|\le\mathfrak c$, then the definition of the Kunen-Shelah property yields the upper bound $$|\ConvH(X)|\le|X^*|^\w\le\mathfrak c^\w=\mathfrak c.$$

If $|\ConvH(X)|\le\mathfrak c$, then $|X^*|\le\mathfrak c$ as $|X^*|\le|\ConvH(X)|$ (because each functional $f\in X^*$ can be uniquely identified with its polar half-space $f^{-1}\big((-\infty,1]\big)\in\ConvH(X)$). Assuming that $X$ fails to have the Kunen-Shelah property and applying Theorem~8.19 of \cite{HMVZ} (see also \cite{GJMMP}), we can find a  sequence $\{x_\alpha\}_{\alpha<\w_1}\subset X$ such that for every $\alpha<\w_1$ the point $x_\alpha$ does not lie in the closed convex hull $C_{\w_1\setminus\{\alpha\}}$ of the set $\{x_\alpha\}_{\alpha\in\w_1\setminus\{\alpha\}}$. Now for every subset $A\subset\w_1$ consider the closed convex hull $C_A=\cconv\{x_\alpha\}_{\alpha\in A}$. We claim that $C_A\ne C_B$ for any distinct subsets $A,B\subset\w_1$. Indeed, if $A\ne B$ then the symmetric difference $(A\setminus B)\cup(B\setminus A)$ contains some ordinal $\alpha$. Without loss of generality, we can assume that $\alpha\in A\setminus B$. Then $x_\alpha\in C_A\setminus C_B$ as $C_B\subset C_{\w_1\setminus\{\alpha\}}\ni x_\alpha$. This implies that $\{C_A:A\subset \w_1\}$ is a subset of cardinality $2^{\w_1}>\mathfrak c$ in $\ConvH(X)$ and hence $|\ConvH(X)|\ge 2^{\w_1}>\mathfrak c$, which is a desired contradiction.
\end{proof}

Among the connected components of $\ConvH(X)$  there is a special
one, namely, the component $\mathcal H_0$ containing the singleton
$\{0\}$.
This component coincides with the space $\BConv_H(X)$ of
all non-empty bounded closed convex subsets of a Banach space $X$.
The spaces $\BConv_H(X)$ have been intensively studied both by topologists
\cite{NQS}, \cite{Sak} and analysts \cite{GSM}. In particular, S.Nadler, J.
Quinn and N.M.~Stavrakas \cite{NQS} proved that for a finite $n\ge
2$ the space $\BConv_H(\IR^n)$  is homeomorphic to
$Q\times\bar\IR_+$ while K.~Sakai  proved in \cite{Sak} that
for an infinite-dimensional Banach space $X$ the space $\HH_0=\BConv_H(X)$ is homeomorphic to a non-separable Hilbert space. Moreover, if $X$ is separable or reflexive, then $\dens(\HH_0)=2^{\dens(X)}$.
In the latter case the density $\dens^*(X^*)$ of the dual space $X^*$ in the weak$^*$ topology is equal to the density $\dens(X)$ of $X$. Banach spaces $X$ with $\dens^*(X^*)=\dens(X^*)$ are called DENS Banach spaces, see \cite[5.39]{HMVZ}. By Proposition 5.40 of \cite{HMVZ}, the class of DENS Banach spaces includes all weakly Linedl\"of determined spaces, and hence all weakly countably generated and all reflexive Banach spaces.

Applying Theorem~\ref{main} to describing the topology of the component  $\HH_0=\BConv_H(X)$, we obtain the following classification.

\begin{introcor}\label{cor3} The space $\HH_0=\BConv_H(X)$ of non-empty bounded closed convex subsets of a Banach space $X$ is homeomorphic to one of the spaces: $\{0\}$, $\IR\times\bar\IR_+$, $Q\times \bar\IR_+$ or the Hilbert space $l_2(\kappa)$ of density $\kappa\ge\mathfrak c$. More precisely, $\BConv(X)$ is homeomorphic to:
\begin{enumerate}
\item $\{0\}$ iff $\dim X=0$;
\item $\IR\times\IR^+$ iff $\dim X=1$;
\item $Q\times\bar\IR_+$ iff $2\le\dim(X)<\infty$;
\item $l_2(\kappa)$ for some cardinal $\kappa\in [2^{\dens^*(X^*)},2^{\dens(X)}]$ iff $\dim(X)=\infty$;
\item $l_2(2^{\dens(X)})$ if $X$ is an infinite-dimensional DENS Banach space.
\end{enumerate}
\end{introcor}

\begin{proof} This corollary will follow from Theorem~\ref{main} as soon as we check that $$2^{\dens^*(X^*)}\le\dens(\HH_0)\le |\HH_0|\le |\ConvH(X)|\le 2^{\dens(X)}$$
for each infinite-dimensional Banach space $X$.

In fact, the inequality $|\ConvH(X)|\le 2^{\dens(X)}$ has general-topological nature and follows from the known fact that the number of closed subsets (equal to the number of open subsets) of a topological space $X$ does not exceed $2^{w(X)}$.

To prove that $2^{\dens^*(X^*)}\le\dens(\HH_0)$ we shall use a result of Plichko \cite{Pli80} (see also Theorem~4.12 \cite{HMVZ}) saying that for each infinite-dimensional Banach space $X$ there is a bounded sequence $\{(x_\alpha,f_\alpha)\}_{\alpha<\kappa}\subset X\times X^*$ of length $\kappa=\dens^*(X^*)$, which is biorthogonal in the sense that $f_\alpha(x_\alpha)=1$ and $f_\alpha(x_\beta)=0$ for any distinct ordinals $\alpha,\beta<\kappa$. Let $L=\sup\{\|x_\alpha\|,\|f_\alpha\|:\alpha<\kappa\}$.

For every subset $A\subset\kappa$ consider the closed convex hull
$C_A=\cconv(\{x_\alpha\}_{\alpha\in A})$ of the set $\{x_\alpha\}_{\alpha\in A}$. We claim that for any distinct subsets $A,B\subset\kappa$ we get $\dH(C_A,C_B)\ge \frac1{L}$.
Indeed, since $A\ne B$ the symmetric difference $(A\setminus B)\cup(B\setminus A)$ contains some ordinal $\alpha$. Without loss of generality, we can assume that $\alpha\in A\setminus B$.
Then $C_B\subset f^{-1}(0)$ and hence for each $c\in C_B$ we get $\|x_\alpha-c\|\ge\frac{|f_\alpha(x_\alpha)-f_\alpha(c)|}{\|f_\alpha\|}\ge \frac{|1-0|}{L}$, which implies $\dist(x_\alpha,C_B)\ge\frac1{L}$ and hence $\dH(C_A,C_B)\ge\frac1L$ as $x_\alpha\in C_A$.

Now we see that $\C=\{C_A:A\subset\kappa\}$ is a closed discrete subspace in $\HH_0$ and hence  $\dens(\HH_0)\ge|\C|=2^\kappa=2^{\dens^*(X*)}$.
\end{proof}

Corollaries~\ref{cor1} and \ref{cor2} motivate the following problem.

\begin{problem} Is $|X^*|\le\mathfrak c$ for each Banach space $X$ with the Kunen-Shelah property?
\end{problem}

Another problem concerns possible densities of the components of the space $\ConvH(X)$.

\begin{problem}\label{pr1.2} Let $X$ be an infinite-dimensional Banach space.
Is it true that each component $\HH$ (in particular, the component $\HH_0$) of the space $\ConvH(X)$ has density $2^\kappa$ or $2^{<\kappa}=\sup\{2^\lambda:\lambda<\kappa\}$ for some cardinal $\kappa$?
\end{problem}
 Observe that under GCH (the Generalized Continuum Hypothesis) the answer to Problem~\ref{pr1.2} is trivially ``yes'' as under GCH all cardinals are of the form $2^{<\kappa}$ for some $\kappa$.
\medskip

\section{Hypermetric spaces}\label{s2}

Because the Hausdorff distance $\dH$ on $\ConvH(X)$ can
take the infinite value we should work with generalized metrics called
hypermetrics. The precise definition is as follows:

A {\em hypermetric} on a set $X$ is a function $d:X\times
X\to[0,\infty]$ satisfying the three axioms of a usual metric:
\begin{itemize}
\item  $d(x,y)=0$ iff $x=y$,
\item  $d(x,y)=d(y,x)$,
\item  $d(x,z)\le d(x,y)+d(y,z)$.
\end{itemize}
Here we extend the addition operation from $(-\infty,\infty)$ to $[-\infty,\infty]$ letting
$$\infty+\infty=\infty,\;-\infty+(-\infty)=-\infty,\;\infty+(-\infty)=-\infty+\infty=0$$and
$$x+\infty=\infty+x=\infty,\; x+(-\infty)=-\infty+x=-\infty$$ for every $x\in (-\infty,\infty)$.

A {\em hypermetric space} is a pair $(X,d)$ consisting of a set $X$ and a hypermetric $d$ on $X$.
It is clear that each metric is a hypermetric and hence each metric space is a hypermetric space.

In some respect, the notion of a hypermetric is more convenient
than the usual notion of a metric. In particular, for any family
$(X_i,d_i)$, $i\in\mathcal I$, of hypermetric spaces it is trivial
to define a nice hypermetric $d$ on the topological sum
$X=\oplus_{i\in\I}X_i$. Just let
$$d(x,y)=\begin{cases}
d_i(x,y),&\mbox{if $x,y\in X_i$};\\
\infty,&\mbox{otherwise}.
\end{cases}
$$
The obtained hypermetric space $(X,d)$ will be called the {\em direct sum} of the family of hypermetric spaces $(X,d_i)$, $i\in\I$.

In fact, each hypermetric space $(X,d)$ decomposes into the direct
sum of metric subspaces of $X$ called metric components of $X$.
More precisely, a {\em metric component} of $X$ is an equivalence
class of $X$ by the equivalence relation $\sim$ defined as $x\sim
y$ iff $d(x,y)<\infty$. So, the {\em metric component} of a point
$x\in X$ coincides with the set $\IB_{<\infty}(x)=\{x'\in
X:d(x,x')<\infty\}$. The restriction of the hypermetric $d$ to
each metric component is a metric. Therefore $X$ is the direct sum
of its metric components  and hence understanding the
(topological) structure of a hypermetric space reduces to studying
the metric (or topological) structure of its metric components.

A typical example of a hypermetric is the Hausdorff hypermetric $\dH$ on the space $\Cld(X)$ of non-empty closed subsets of a (linear) metric space $X$ (and the restriction of $\dH$ to the subspace $\Conv(X)\subset \Cld(X)$ of non-empty closed convex subsets of $X$). So both $\Cld_{\mathsf H}(X)=(\Cld(X),\dH)$ and $\ConvH(X)=(\Conv(X),\dH)$ are hypermetric spaces.

A much simple (but still important) example of a hypermetric space is the extended real line $\overline{\IR}=[-\infty,\infty]$ with the hypermetric
$$d(x,y)=
\begin{cases}
|x-y|,&\mbox{if $x,y\in (-\infty,\infty)$,}\\
0,&\mbox{if $x=y\in\{-\infty,\infty\}$,}\\
\infty,&\mbox{otherwise},
\end{cases}
$$
which will be denoted by $|x-y|$.
The hypermetric space $\overline{\IR}$ has three metric
components: $\{-\infty\}$, $\IR$, $\{\infty\}$.
\smallskip

This example allows us to construct another important example of a hypermetric space.
Namely, for a set $\Gamma$ consider the space $\overline \IR^\Gamma$ of functions from $\Gamma$ to $\overline\IR$ endowed with the hypermetric
$$d(f,g)=\|f-g\|_\infty=\sup_{\gamma\in\Gamma}|f(\gamma)-g(\gamma)|.$$
The obtained hypermetric space
$(\overline{\IR}^\Gamma,\|\cdot-\cdot\|_\infty)$ will be denoted
by $\bar l_\infty(\Gamma)$. Observe that the topology of $\bar
l_\infty(\Gamma)$ is different from the Tychonoff product topology
of $\overline{R}^\Gamma$. Another reason for using the notation
$\bar l_\infty(\Gamma)$ is that the metric component of $\bar
l_\infty(\Gamma)$ containing the zero function
 coincides with the classical Banach space $l_\infty(\Gamma)$ of bounded functions on $\Gamma$.
 More generally, for each $f_0\in\bar l_\infty(\Gamma)$ its metric component
$$\IB_{<\infty}(f_0)=\{f\in\bar l_\infty(\Gamma):\|f-f_0\|_\infty<\infty\}$$
is isometric to the Banach space $l_\infty(\Gamma_0)$ where $\Gamma_0=\{\gamma\in\Gamma:|f_0(\gamma)|<\infty\}$.

It turns out that for every normed space $X$ the space $\ConvH(X)$ nicely embeds into the hypermetric space $\bar l_\infty(\IS^*)$ where $$\IS^*=\{x^*\in X^*:\|x^*\|=1\}$$ stands for the unit sphere of the dual Banach space $X^*$.

Namely, consider the function
$$\delta:\ConvH(X)\mapsto \bar l_\infty(\IS^*),\;\delta:C\mapsto\delta_C$$
where $\delta_C(x^*)=\sup x^*(C)$ for $x^*\in\IS^*$. The function
$\delta$ will be called the {\em canonical representation} of
$\ConvH(X)$.

\begin{proposition}\label{p2.1} For every normed space $X$ the canonical representation $\delta:\ConvH(X)\to\bar l_\infty(\IS^*)$ is an isometric embedding.
\end{proposition}

\begin{proof} Let $A,B\in\ConvH(X)$ be two convex sets. We should prove that $\dH(A,B)=\|\delta_A-\delta_B\|$, where $$\|\delta_A-\delta_B\|=\sup_{x^*\in\IS^*}|\delta_A(x^*)-\delta_B(x^*)|=\sup_{x^*\in\IS^*}|\sup x^*(A)-\sup x^*(B)|.$$

The inequality $\|\delta_A-\delta_B\|\le \dH(A,B)$ will follow as soon as we check that
$|\sup x^*(A)-\sup x^*(B)|\le \dH(A,B)$ for each functional $x^*\in\IS^*$. This is trivial if $\dH(A,B)=\infty$.  So we assume that $\dH(A,B)<\infty$.
To obtain a contradiction, assume that $|\sup x^*(A)-\sup x^*(B)|> \dH(A,B)$. Then either $\sup x^*(A)-\sup x^*(B)> \dH(A,B)$ or
$\sup x^*(B)-\sup x^*(A)> \dH(A,B)$. In the first case
$\sup x^*(B)\ne\infty$, so we can find a point $a\in A$ with $x^*(a)-\sup x^*(B)> \dH(A,B)$. It follows from the definition of
the Hausdorff metric $\dH(A,B)\ge \dist(a,B)$ that $\|a-b\|<x^*(a)-\sup x^*(B)$ for some point $b\in B$. Then $x^*(a)-x^*(b)\le \|x^*\|\cdot\|a-b\|<x^*(a)-\sup x^*(B)$ and hence $x^*(b)>\sup x^*(B)$, which is a contradiction.

By analogy, we can derive a contradiction from the assumption $\sup x^*(B)-\sup x^*(A)> \dH(A,B)$ and thus prove the inequality
$\|\delta_A-\delta_B\|\le \dH(A,B)$.
\smallskip

To prove the reverse inequality $\|\delta_A-\delta_B\|\ge \dH(A,B)$ let us consider two cases:

(i) $\dH(A,B)=\infty$. To prove that $\infty=\|\delta_A-\delta_B\|$, it suffices given any number $R<\infty$ to find a linear functional $x^*\in\IS^*$ such that $|\sup x^*(A)-\sup x^*(B)|\ge R$.

The equality $\dH(A,B)=\infty$ implies that either $\sup_{a\in A}\dist(a,B)=\infty$ or $\sup_{b\in B}\dist(b,A)=\infty$. In the first case we can find a point $a\in A$ with $\dist(a,B)\ge R$ and using the Hahn-Banach Theorem construct a linear functional $x^*\in \IS^*$ that separates the convex set $B$ from the closed $R$-ball $\bar B(a,R)=\{x\in X:\|x-a\|\le R\}$ in the sense that $\sup x^*(B)\le \inf x^*(\bar B(a,R))$. For this functional $x^*$ we get $\sup x^*(A)\ge x^*(a)\ge R+\inf x^*(\bar B(a,R))\ge R+\sup x^*(B)$ and thus $\sup x^*(A)-\sup x^*(B)\ge R$.

In the second case, we can repeat the preceding argument to find a linear functional $x^*\in \IS^*$ with $$|\sup x^*(A)-\sup x^*(B)|\ge \sup x^*(B)-\sup x^*(X)\ge R.$$

(ii) $\dH(A,B)<\infty$. To prove that $\dH(A,B)\ge \|\delta_A-\delta_B\|$ it suffices given any number $\e>0$ to find a linear functional $x^*\in\IS^*$ such that $|\sup x^*(A)-\sup x^*(B)|\ge \dH(A,B)-\e$. It follows from the definition of $\dH(A,B)$ that either there is a point $a\in A$ with $\dist(a,B)>\dH(A,B)-\e$ or else there is a point $b\in B$ with $\dist(b,A)>\dH(A,B)-\e$. In the first case we an use the Hahn-Banach Theorem to find a linear functional $x^*\in \IS^*$ which separates the convex set $B$ from the closed  $R$-ball $\bar B(a,R)$ where $R=\dH(A,B)-\e$ in the sense that  $\sup x^*(B)\le \inf x^*(\bar B(a,R)$. Then
$$\sup x^*(B)\le\inf x^*(\bar B(a,R))=x^*(a)-R\le\sup x^*(A)-R$$ and hence $$|\sup x^*(A)-\sup x^*(B)|\ge \sup x^*(A)-\sup x^*(B)\ge R=\dH(A,B)-\e.$$
The second case can be considered by analogy.
\end{proof}

\section{Assigning cones to components of $\ConvH(X)$}\label{s3}

In this section to each convex set $C$ of a normed space $X$ we assign two cones: the characteristic cone  $V_C\subset X$ and the dual characteristic cone $V_C^*\subset X^*$.

We recall that a subset $V$ of a linear space $L$ is called a {\em convex cone} if  $ax + by \in V$ for any points $x, y \in W$ and non-negative real numbers $a, b\in[0,+\infty)$.

For a convex subset $C$ of a normed space $X$ its {\em characteristic cone} of $C$  is the convex cone
$$V_C=\{v\in X:\forall c\in C,\;\;c+\bar\IR_+v\subset c\}$$lying in the normed space $X$, and its {\em dual characteristic cone} $V^*_C$ is a closed convex cone
$$V_C^*=\{x^*\in X^*:\sup x^*(C)<\infty\}$$which lies in the dual Banach space $X^*$.

It turns out that the characteristic cone $V_C$ of a convex set $C$ is uniquely determined by its dual characteristic cone $V^*_C$.

\begin{lemma}\label{l3.1} For any non-empty closed convex set $C$ in a normed space $X$ we get
$$V_C=\bigcap_{f\in V^*_C}f^{-1}\big((-\infty,0]\big).$$
\end{lemma}

\begin{proof} Fix any vector $v\in V_C$ and a functional $f\in V^*_C$. Observe that for each  number $t\in\bar\IR_+$, we get $c+tv\in C$ and hence $f(c)+tf(v)\le\sup f(C)<\infty$, which implies that $f(v)\le 0$. This proves the inclusion $V_C\subset\bigcap_{f\in V^*_C}f^{-1}\big((-\infty,0]\big).$

To prove the reverse inclusion, fix any vector $v\in X\setminus V_C$. Then for some point $c\in C$ and some positive real number $t$ we get $c+tv\notin C$. Using the Hahn-Banach Theorem, find a functional $f\in X^*$ that separates the convex set $C$ and the point $x=c+tv$ in the sense that $\sup f(C)<f(c+tv)$. Then $f(c)\le\sup f(C)<f(c)+tf(v)$ implies that $f(v)>0$ and $v\notin f^{-1}\big((-\infty,0]\big)$.
\end{proof}

Let $X$ be a normed space. It is easy to see that for each component $\HH$ of the space $\ConvH(X)$
and any two convex sets $A,B\in\HH$ we get $C^*_A=C^*_B$. In this case Lemma~\ref{l3.1} implies that  $C_A=C_B$ as well. This allows us to define the {\em characteristic cone} $V_\HH$ and the {\em dual characteristic cone} $V^*_\HH$ of the component $\HH$ letting $V_\HH=V_C$ and $V^*_\HH=V^*_C$ for any convex set $C\in\HH$. Lemma~\ref{l3.1} guarantees that
$$V_\HH=\bigcap_{f\in V^*_\HH}f^{-1}\big((-\infty,0]\big),$$
so the characteristic cone $V_\HH$ of $\HH$ is uniquely determined by its dual characteristic cone $V^*_\HH$.

\section{The algebraical structure of $\ConvH(X)$}\label{s4}

In this section given a normed space $X$ we study the algebraic properties of the canonical representation $\delta:\ConvH(X)\to \bar l_\infty(\IS^*)$.

Note that the space $\ConvH(X)$ has a rich algebraic structure: it possesses three interrelated algebraic operations: the
multiplication by a real number, the addition, and taking the maximum. More precisely, for a real number $t\in\IR$, and convex sets $A,B\in\ConvH(X)$ let
\begin{itemize}
\item[ ] $t\cdot A=\{ta:a\in A\}$;
\item[ ] $A\oplus B=\overline{A+B}$;
\item[ ] $\max\{A,B\}=\oco(A\cup B)$, where
\item[ ] $\oco(Y)$ stands for the closed convex hull of a subset $Y\subset X$.
\end{itemize}
\smallskip

The hypermetric space $\overline\IR$ also has the corresponding
three operations (multiplication by a real number, addition and
taking maximum) which induces the tree operations on $\bar
l_\infty(\Gamma)=\overline{\IR}^\Gamma$.

\begin{proposition}\label{p4.1} The canonical representation
$\delta:\ConvH(X)\to\bar l_\infty(\IS^*)$ has the following properties:
\begin{enumerate}
\item $\delta(A\oplus B)=\delta(A)+\delta(B)$;
\item $\delta(\max\{A,B\})=\max\{\delta(A),\delta(B)\}$;
\item $\delta(rA)=r\delta(A)$;
\end{enumerate}
for every non-negative real number $r$ and convex sets
$A,B\in\ConvH(X)$.
\end{proposition}

\begin{proof} The three items of the proposition follow from the three obvious equalities
$$
\begin{aligned}
&\sup x^*(A\oplus B)=x^*(A+B)=x^*(A)+x^*(B),\\
&\sup x^*(\oco(A\cup B))=\sup x^*(A\cup B)=\max\{\sup x^*(A),\sup x^*(B)\},\\
&\sup x^*(rA)=r\sup x^*(A).
\end{aligned}
$$
holding for every functional $x^*\in X^*$.
\end{proof}

\begin{remark} Easy examples show that the last item of Proposition~\ref{p4.1} does not hold for negative real numbers $r$. This means that the operator \mbox{$\delta:\ConvH(X)\to\bar l_\infty(\IS^*)$} is positively homogeneous but not homogeneous.
\end{remark}

The operations of addition and multiplication by a real number
allow us to define another important operation on $\ConvH(X)$
preserved by the canonical representation $\delta$, namely the
{\em Minkovski operation}
$$\mu:\ConvH(X)\times\ConvH(X)\times[0,1]\to\ConvH(X),\;\mu:(A,B,t)\mapsto (1-t)A\oplus tB$$
of producing a convex combination.
Proposition~\ref{p4.1} implies that the canonical representation $\delta$ is {\em affine} in the sense that $$\delta\big((1-t)A\oplus tB\big)=(1-t)\delta(A)+t\delta(B)$$ for every $A,B\in\ConvH(X)$ and $t\in[0,1]$.

Propositions~\ref{p2.1} and \ref{p4.1} will help us to establish the metric properties of the algebraic operations on $\ConvH(X)$.

\begin{proposition}\label{p4.2} Let $A,B,C,A',B'\in\ConvH(X)$ be five convex sets and $r\in\IR$ and $t,t'\in[0,1]$ be three real numbers.
Then
\begin{enumerate}
\item $\dH(A\oplus B,A'\oplus B')\le \dH(A,A')+\dH(B,B')$;
\item $\dH(A\oplus B,A\oplus C)=\dH(B,C)$ provided $V^*_A\supset V^*_B\cup V^*_C$;
\item $\dH(\max\{A,B\},\max\{A',B'\})\le \max\{\dH(A,A'),\dH(B,B')\}$;
\item $\dH(r\cdot A,r\cdot B)=|r|\cdot \dH(A,B)$;
\item $\dH((1-t)A\oplus tB,(1-t')A\oplus t'B)=|t-t'|\dH(A,B)$.
\end{enumerate}
\end{proposition}

\begin{proof} All the items easily follow from Propositions~\ref{p2.1}, \ref{p4.1}, and metric properties of algebraic operations on the hypermetric space $\bar l_\infty(\IS^*)$.
\end{proof}

Observe that the metric components of the hypermetric space $\bar l_\infty(\IS^*)$ are closed with respect to taking maximum and producing a convex combination. Moreover those operations are continuous on metric components of $\bar l_\infty(\IS^*)$.
With help of the canonical representation those properties of $\bar l_\infty(\IS^*)$ transform into the corresponding properties of $\ConvH(X)$. In such a way we obtain

\begin{corollary}\label{c4.3} Each metric component $\mathcal H$ of $\ConvH(X)$ is closed
under the operations of taking maximum and producing a convex combination.
Moreover those operations are continuous on $\mathcal H$.
\end{corollary}

\begin{corollary}\label{c4.4} Each metric component $\mathcal H$ of $\ConvH(X)$ is isometric to a convex max-subsemilattice of the Banach lattice $l_\infty(\IS^*)$.
\end{corollary}

A subset of a Banach lattice is called a {\em max-subsemilattice} is it is closed under the operation of taking maximum.

By a recent result of Banakh and Cauty \cite{BC}, each non-locally compact closed convex subset of a Banach space is homeomorphic to an infinite-dimensional Hilbert space. This result combined with Corollary~\ref{c4.4} implies:

\begin{corollary}\label{c4.5} Let $X$ be a Banach space. A metric component $\mathcal H$ of $\ConvH(X)$ is homeomorphic to an infinite-dimensional Hilbert space if and only if $\mathcal H$ is not locally compact.
\end{corollary}

This corollary reduces the problem of recognition of the topology of non-locally compact components of $\ConvH(X)$ to calculating their densities. This problem was considered in \cite{BH2} and \cite{BH3}). In particular, \cite{BH2} contains the following characterization:

\begin{proposition}\label{p4.7} For a Banach space $X$ and a metric component $\HH$ of the space $\ConvH(X)$ the following conditions are equivalent:
\begin{enumerate}
\item $\HH$ is separable;
\item $\dens(\HH)<\mathfrak c$;
\item $\HH$ contains a polyhedral convex set;
\item the characteristic cone $V_\HH$ is polyhedral and belongs to $\HH$;
\end{enumerate}
\end{proposition}

\begin{remark}\label{r4.6} Each metric component of $\ConvH(X)$ being homeomorphic to a convex set, is connected and thus coincides with a connected component of $\ConvH(X)$. Hence there is no difference between metric and connected components of $\ConvH(X)$, so using the
term {\em component} of $\ConvH(X)$ (without an adjective ``metric'' or ``connected'') will not lead to misunderstanding.
\end{remark}

\section{Operators between spaces of convex sets}\label{s5}

Each linear continuous operator $T:X\to Y$ between normed spaces induces a map $\overline{T}:\ConvH(X)\to\ConvH(Y)$ assigning to each closed convex set $A\in\ConvH(X)$ the closure $\overline{T(A)}$ of its image $T(A)$ in $Y$.
In this section we study properties of the induced operator $\overline{T}$.
We start with algebraic properties that trivially follow from the linearity and continuity of the operator $T$.

\begin{proposition}\label{p5.1} If $T:X\to Y$ is a linear continuous operator between Banach spaces  and $\overline{T}:\ConvH(X)\to\ConvH(X)$ is the induced operator, then

\begin{enumerate}
\item $\overline{T}(\max\{A,B\})=\max\{\overline{T}(A),\overline{T}(B)\}$;
\item $\overline{T}(r\cdot A)=r\cdot \overline{T}(A)$;
\item $\overline{T}(A\oplus B)=\overline{T}(A)\oplus\overline{T(B)}$;
\item $\overline{T}((1-t)A\oplus tB)=(1-t)\overline{T}(A)\oplus t\overline{T(B)}$;
\end{enumerate}
 for any sets $A,B\in\ConvH(X)$ and real numbers $r\in\IR$ and $t\in[0,1]$.
\end{proposition}

We shall be mainly interested in the operators $\overline{T}$ induced by quotient operators $T$.
We recall that for a closed linear subspace $Z$ of a normed space $X$ the quotient normed space $X/Z=\{x+Z:x\in X\}$ carries the quotient norm
$$\|x+Z\|=\inf_{y\in x+Z}\|y\|.$$
By $q:X\to X/Z$, $q:x\mapsto x+Z$ we shall denote the quotient operator and by $\bar q:\ConvH(X)\to\ConvH(X/Z)$ the induced operator between the spaces of closed convex sets.

For a closed convex set $C\subset X$ by $C/Z$ we denote the image $q(C)\subset X/Z$. So, $\bar q(C)=\overline{C/Z}$. If $Z\subset V_C$, then the set $C/Z$ is closed in $X/Z$ and hence $\bar q(C)=C/Z$. Indeed, $Z\subset V_C$ implies that $C+Z=C$ and hence $C/Z=(X/Z)\setminus q(X\setminus C)$ is closed in $X/Z$ being the complement of the set $q(X/\setminus C)$ which is open as the image of the open set $X\setminus C$ under the open map $q:X\to X/Z$.

We shall need the following simple reduction lemma:

\begin{lemma}\label{l5.2} Let $Z$ be a closed linear subspace of a normed space $X$ and $A,B$ are non-empty closed convex subsets of $X$. If $Z\subset V_A\cap V_B$, then $\dH(A,B)=\dH(A/Z,B/Z)$.
\end{lemma}

\begin{proof} The inequality $\dH(A/Z,B/Z)\le \dH(A,B)$ follows from $\|q\|\le 1$. Assuming that
$\dH(A/Z,B/Z)<\dH(A,B)$, we can find a point $a\in A$ with $\dist(a,B)>\dH(A/Z,b/Z)$ or a point  $b\in B$ with $\dist(b,A)>\dH(A/Z,B/Z)$. Without loss of generality, we can assume that $\dist(a,B)>\dH(A/Z,B/Z)$ for some point $a\in A$. Consider its image $a'=q(a)\in A/Z$ under the quotient operator $q:X\to X/Z$. By the definition of the Hausdorff metric, $\dH(A/Z,B/Z)<\dist(a,B)$, there is a point $b'\in B/Z$ such that $\|b'-a'\|<\dist(a,B)$. By the definition of the quotient norm, there is a vector $z\in Z$ such that $q(z)=b'-a'$ and $\|z\|<\dist(a,B)$. Now consider the
point $b=a+z$ and observe that  $q(b)=q(a)+q(z)=a'+b'-a'=b'\in B/Z$ and hence $b\in q^{-1}(B/Z)=B+Z=B$. So, $\dist(a,B)\le\|a-b\|=\|z\|<\dist(a,B)$, which is a desired contradiction that completes the proof of the equality $\dH(A,B)=\dH(A/Z,B/Z)$.
\end{proof}

\begin{corollary}\label{c5.3} Let $X$ be a normed space $X$, $\HH$ be a component of the space $\ConvH(X)$, and $Z$ be a closed linear subspace of $X$. If $Z\subset V_\HH$, then the quotient operator
$$\bar q:\HH\to\HH/Z,\;\;\bar q:C\mapsto C/Z,$$ maps isometrically the component $\HH$ of $\ConvH(X)$ onto the component $\HH/Z$ of $\ConvH(X/Z)$ containing some (equivalently, each) convex set $C/Z$ with $C\in\HH$.
\end{corollary}

\section{Proof of Theorem~\ref{main}}\label{s7}

Let $X$ be a Banach space and $\HH$ be a component of the space $\ConvH(X)$.

If $\HH$ contains no polyhedral convex set, then by Propositions~\ref{p4.7}, it has density $\dens(\HH)\ge\mathfrak c$. Consequently, $\HH$ is not locally compact and by
Corollary~\ref{c4.5}, $\HH$ is homeomorphic to the non-separable Hilbert space $l_2(\kappa)$ of density $\kappa=\dens(\HH)\ge\mathfrak c$.

It remains to analyze the topological
structure of $\HH$ if it contains a polyhedral convex set. In this case Proposition~\ref{p4.7} guarantees that the characteristic cone $V_\HH$ belongs to $\HH$ and is polyhedral in $X$. If $V_\HH=X$, then $\HH=\{X\}$ is a singleton. So, we assume that $V_\HH\ne X$. Since the cone $V_\HH$ is polyhedral, the the closed linear subspace $Z=-V_{\HH}\cap V_{\HH}$ has finite codimension in $X$. Then the quotient Banach space $\tilde X=X/Z$ is finite-dimensional. Let $q:X\to \tilde X$ be the quotient operator.

By Corollary~\ref{c5.3}, the component $\HH$ is isometric to the component $\tilde\HH=\HH/Z$ of the space $\ConvH(\tilde X)$ of closed convex subsets of the finite-dimensional Banach space $\tilde X$. The component $\tilde\HH$ contains the polyhedral convex cone
$V_{\tilde\HH}=q(V_{\HH})$ which has the property $-V_{\tilde \HH}\cap V_{\tilde \HH}=\{0\}$.

The cone $V_{\tilde \HH}$ can be of two types.
\smallskip

1. The cone $V_{\tilde \HH}=\{0\}$ is trivial. In this case $\HH$ contains the closed linear subspace $Z=V_\HH$ of finite codimension in $X$. Taking into account that $V_\HH\ne X$, we conclude that $\dim(\tilde X)\ge 1$. Depending on the value of $\dim(\tilde X)$, we have two subcases.
\smallskip

1a. The dimension $\dim(\tilde X)=1$ and hence $\HH$ contains the linear subspace $Z=V_\HH$ of codimension 1 in $X$.
In this case $\tilde \HH$ coincides with the space $\BConv_H(\tilde X)$ of non-empty bounded closed convex subsets of the one-dimensional Banach space $\tilde X$ and hence $\tilde \HH$ is isometric to the half-plane $\IR\times\bar\IR_+$.
\smallskip

1b. The dimension $\dim(\tilde X)\ge 2$ and hence $\HH$ contains the linear subspace $Z$ of codimension $\ge 2$ in $X$. In this case $\tilde \HH$ coincides with the space $\BConv_H(\tilde X)$ of non-empty bounded closed convex subsets of the finite-dimensional Banach space $\tilde X$ of finite dimension $\dim(\tilde X)\ge 2$. By the result of Nadler, Quinn and Stavrakas \cite{NQS}, the space $\BConv_H(\tilde X)$ is homeomorphic to the Hilbert cube manifold $Q\times\bar\IR_+$.
\medskip

2. The characteristic cone $V_{\tilde \HH}\ne\{0\}$ is not trivial. This case has two subcases.
\smallskip

2a. $\dim(\tilde X)=\dim(V_{\tilde\HH})=1$. In this case the component $\HH$ (and its isometric copy $\HH$) is isometric to the real line $\IR$.

2b. $\dim(\tilde X)\ge 2$. In this case we shall prove that the component $\tilde \HH$ (and its isometric copy $\HH$) is homeomorphic to the separable Hilbert space $l_2$. This will follow from the separability of $\HH$ and Corollary~\ref{c4.5} as soon as we check that the space $\tilde \HH$ is not locally compact. To prove this fact, it suffices for every positive $\e<1$ to construct a sequence of closed convex sets $\{C_n\}_{n\in\IN}\subset\tilde \HH$ such that $\dH(C_n,V_{\tilde \HH})\le\e$ and $\inf_{n\ne m}\dH(C_n,C_m)>0$.

The cone $V_{\tilde \HH}$ is polyhedral and hence is generated by for some finite set $E\subset \tilde X\setminus \{0\}$, see \cite{Klee} or Theorem 1.1 of \cite{Ziegler}. For every $e\in E$ the vector  $-e$ does not belong to $V_{\tilde \HH}$. Then the Hahn-Banach Theorem yields a linear functional $h_e\in X^*$ such that $h_e(-e)<\inf h_e(V_{\tilde\HH})=0$. It can be shown that the functional $h=\sum_{e\in E}h_e$ has the property $h(v)>0$ for all $v\in V_{\tilde\HH}\setminus\{0\}$.

Since $\dim(\tilde X)\ge 2$ and $V_{\tilde \HH}\ne \tilde X$, we can find a non-zero linear continuous functional $f:\tilde X\to \IR$ such that $\sup f(V_{\tilde \HH})=0$ and the intersection
$f^{-1}(0)\cap V_{\tilde\HH}$ contains a non-zero vector $x\in \tilde X$. Multiplying $x$ by a suitable positive constant, we can assume that $h(x)=1$. Since $h^{-1}(0)\cap V_{\tilde \HH}=\{0\}\ne f^{-1}(0)\cap V_{\tilde \HH}$, the functionals $h$ and $f$ are distinct and hence there is a vector $y\in h^{-1}(0)\setminus f^{-1}(0)$ with norm $\|y\|=\e$. Replacing $y$ by $-y$, if necessary, we can assume that $f(y)>0$.

For every $n\in\w$ consider the point $c_n=3^nx+y$ and the closed convex set $$C_n=\max\big\{V_{\tilde \HH},\{c_n\}\big\}=\cconv(V_{\tilde \HH}\cup \{c_n\})\subset \tilde X.$$ It follows from $x\in V_{\tilde \HH}$ and $\dist(c_n,V_{\tilde \HH})\le\dist(3^nx+y,3^nx)=\|y\|=\e$ that $\dH(C_n,V_{\tilde \HH})\le\e$.

We claim that $\inf\limits_{n\ne m}\dH(C_n,C_m)\ge\delta$ where $$\delta=\frac12 f(y)\le \frac12\|y\|=\frac12\e<\frac12.$$
This will follow as soon as we check that $\dist(c_n,C_m)\ge\delta$ for any numbers $n<m$.

Assuming conversely that $\dist(c_n,C_m)<\delta$ and taking into account that the convex set $\conv(V_{\tilde \HH}\cup\{c_m\})$ is dense in $C_m$, we can find a point $c\in \conv(V_{\tilde \HH}\cup\{c_m\})$ such that $\dist(c_n,c)<\delta$. The point $c$ belongs to the convex hull of the set $V_{\tilde \HH}\cup\{c_m\}$ and hence can be written as a convex combination $c=tc_m+(1-t)v=t(3^mx+y)+(1-t)v$ for some $t\in[0,1]$ and $v\in V_{\tilde \HH}$. Observe that $h(c_n)=h(3^nx+y)=3^nh(x)+h(y)=3^m\cdot 1+0=3^n$ while  $h(c)=th(c_m)+(1-t)h(v)\ge th(c_m)=3^mt$. Then
$$3^mt-3^n\le h(c)-h(c_n)\le|h(c)-h(c_n)|\le\|h\|\cdot\|c-c_n\|<1\cdot\delta$$
and hence $$t<3^{n-m}+3^{-m}\delta\le \frac13+\frac13\delta<\frac13+\frac16=\frac12.$$

Next, we apply the functional $f$ to the points $c_n$ and $c$. Since $f(x)=0$, we get $f(c_n)=f(3^nx+y)=f(y)=2\delta$. On the other hand, $f(V_{\tilde \HH})\subset(-\infty,0]$ implies $f(v)\le 0$ and hence
$$f(c)=f(tc_m+(1-t)v)=tf(3^mx+y)+(1-t)f(v)=tf(y)+(1-t)f(v)\le tf(y)=2\delta t.$$
Then $$\delta=2\delta(1-\tfrac12)<2\delta(1-t)\le |f(c_n)-f(c)|\le\|f\|\cdot\|c_n-c\|<\delta,$$ which is a desired contradiction.

The above proof can be visualized in the following picture:

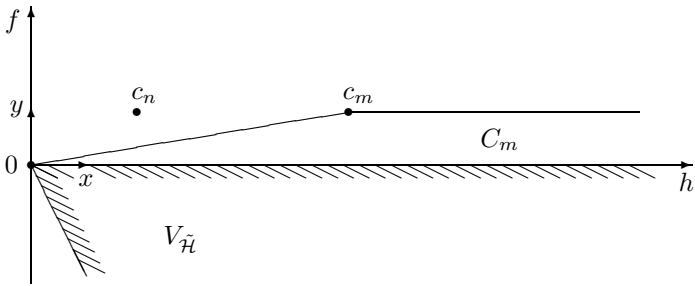
\begin{picture}(300,130)(-50,40)
\put(-10,97){$0$}
\put(0,100){\vector(1,0){250}}
\put(0,100){\vector(1,0){22}}
\put(18,92){$x$}
\put(245,90){$h$}
\put(0,100){\circle*{3}}
\put(0,100){\vector(0,1){60}}
\put(0,100){\line(0,-1){45}}
\put(-10,152){$f$}
\put(0,100){\vector(0,1){22}}
\put(-8,120){$y$}
\put(0,100){\line(6,1){120}}

\put(40,120){\circle*{3}}
\put(38,125){$c_n$}
\put(120,120){\circle*{3}}
\put(120,120){\line(1,0){110}}
\put(118,125){$c_m$}

\put(0,100){\line(1,-2){21}}
\multiput(0,100)(6,0){2}{\line(2,-1){10}}
\multiput(22,100)(6,0){35}{\line(2,-1){10}}
\multiput(0,100)(2,-4){10}{\line(2,-1){10}}
\put(50,70){$V_{\tilde \HH}$}
\put(170,107){$C_m$}
\end{picture}


\newpage
\begin{thebibliography}{}

\bibitem{BC} T.~Banakh, R.~Cauty, {\em Topological classification of convex sets in Frechet spaces}, Studia Math. {\bf 205} (2011), 1--11.

\bibitem{BH1} T.~Banakh, I.~Hetman, {\em A ``hidden'' characterization of polyhedral convex sets}, Studia Math. {\bf 206} (2011), 63--74.


\bibitem{BH2} T.~Banakh, I.~Hetman, {\em A ``hidden'' characterization of approximatively polyhedral convex sets in Banach spaces}, preprint (http://arxiv.org/abs/1111.6708).

\bibitem{BH3} T.~Banakh, I.~Hetman, {\em The hiding and polyhedrality numbers of closed convex sets in Banach spaces}, in preparation.


\bibitem{GSM} A.S.~Granero, M.S.~Jimenez, J.P.~Moreno, {\em Convex sets in
Banach spaces and a problem of Rolewicz}, Studia Math. {\bf 129}:1 (1998), 19--29.

\bibitem{GJMMP} A.~Granero, M.~Jimenez, A.~Montesinos, J.~Moreno, A.~Plichko, {\em On the Kunen--Shelah properties in Banach spaces}, Studia Math. {\bf 157}:2 (2003), 97--120.

\bibitem{HMVZ} P.~H\'ajek, S.~Montesinos, J.~Vanderwerff, V.~Zizler, {\em Biorthogonal systems in Banach spaces}, Springer, New York, 2008.

\bibitem{Klee} V.~Klee, {\em Some characterizations of convex polyhedra},
Acta Math. {\bf 102} (1959) 79--107.

\bibitem{NQS}  S.~Nadler, J. Quinn, N.M. Stavrakas,
 {\it Hyperspaces of compact convex sets,}
 Pacific J.\ Math.\ {\bf 83} (1979), 441--462.

\bibitem{Neg84} S.~Negrepontis, {\em Banach Spaces and Topology}, in: Handbook of Set-Theoretic
Topology, North-Holland, 1984, p. 1045--1142.

\bibitem{Pli80} A.~Plichko, {\em Existence of a bounded total biorthogonal system in a
Banach space}, Teor. Funks. Funkt. Anal. Pril. {\bf 33} (1980), 111--118.

\bibitem{Sak} K.~Sakai, {\em
The spaces of compact convex sets and bounded closed convex sets in a Banach space}, Houston J. Math. {\bf 34}:1 (2008), 289--300.

\bibitem{Shel} S.~Shelah, {\em Uncountable constructions for B.A., e.c. groups and Banach spaces}, Israel J. Math. {\bf 51}:4 (1985), 273--297.

\bibitem{Tod06} S.~Todorcevic, {\em Biorthogonal systems and quotient spaces via Baire category methods}, Math. Ann. {\bf 335}:3 (2006), 687--715.

\bibitem{Tor81} H.Toru\'nczyk, {\em  Characterizing Hilbert space topology}, Fund. Math. {\bf 111}:3 (1981), 247--262.

\bibitem{Tor85} H.~Torunczyk, {\em A correction of two papers concerning Hilbert manifolds}, Fund. Math. {\bf 125}:1 (1985), 89--93.

\bibitem{Ziegler} G.~Ziegler, {\em Lectures on polytopes}, Springer-Verlag, New York, 1995.

\end{thebibliography}
\end{document}